\documentclass[11pt]{amsart}
\usepackage{amsmath}
\usepackage[active]{srcltx}
\usepackage{t1enc}
\usepackage[latin2]{inputenc}
\usepackage{verbatim}
\usepackage{amsmath,amsfonts,amssymb,amsthm}
\usepackage[mathcal]{eucal}
\usepackage{enumerate}
\usepackage[centertags]{amsmath}
\usepackage{graphics}

\theoremstyle{plain}
\newtheorem{theorem}{Theorem}
\newtheorem{lemma}{Lemma}

\theoremstyle{remark}

\newtheorem*{W1}{\textbf{Theorem W1}}
\newtheorem*{W2}{\textbf{Theorem W2}}
\newtheorem*{GN1}{\textbf{Theorem GN1}}
\newtheorem*{GN2}{\textbf{Theorem GN2}}
\newtheorem*{SWS}{\textbf{Theorem SWS}}

\newtheorem{example}{Example}

\begin{document}
\author{Ushangi Goginava }
\title[Maximal Operators of Walsh-Nörlund means]{Maximal Operators of Walsh-N%
örlund means On the Dyadic Hardy Spaces}
\address{U. Goginava, Department of Mathematical Sciences \\
United Arab Emirates University, P.O. Box No. 15551\\
Al Ain, Abu Dhabi, UAE}
\email{zazagoginava@gmail.com; ugoginava@uaeu.ac.ae}
\address{˙ }

\begin{abstract}
The presented paper will be proved the necessary and sufficient conditions
in order maximal operator of Walsh-Nörlund means with non-increasing weights
to be bounded from the dyadic Hardy space $H_{p}(\mathbb{I})$\ to the space $%
L_{p}(\mathbb{I})$.
\end{abstract}

\maketitle

\bigskip \footnotetext{%
\par
The author is very thankful to United Arab Emirates University (UAEU) for
the Start-up Grant 12S100.}

\section{\protect\bigskip Introduction}

In 1992 Móricz and Siddiqi investigated the rate of the approximation by Nö%
rlund means of Walsh-Fourier series \cite{Moricz-Siddiqi}. For Nörlund means
with monotone weights they gave a sufficient condition which provide Nörlund
means for having convergence in $L_{p}$ norm ($1\leq p<\infty $) and $C_{W}$
norm.

The result of \cite{Moricz-Siddiqi} was extended by Fridli, Manchanda and
Siddiqi \cite{FMS-08-ASM} for dyadic martingale Hardy spaces and dyadic
homogeneous Banach spaces . Recently, the theorem of Móricz and Siddiqi was
generalized for $\Theta $-means of Walsh-Fourier series in $L_{p}$ spaces ($%
1\leq p<\infty $) and $C_{W}$ \cite{BlahotaNagy2018}.

The theorems mentioned above are related to the approximation of the Nörlund
means which in turn is related to the uniformly boundedness of the
corresponding operators of the Nörlund means. To study of almost everywhere
convergence of Nörlund means is connected to the study of the boundedness of
the maximum operators corresponding to the Nörlund means.

The first result with respect to the a.e. convergence of the Walsh-Fejér
means is due to Fine \cite{fine1955cesaro}. Later, Schipp \cite%
{schipp1975certain} showed that the maximal operator of the Walsh-Fejér
means is of weak type (1, 1), from which the a. e. convergence follows by
standard argument \cite{marcinkiewicz1939summability}. Schipp result implies
by interpolation also the boundedness of $\sup\limits_{n}\left\vert \sigma
_{n}^{1}\right\vert :L_{p}\rightarrow L_{p}\,\left( 1<p\leq \infty \right) $%
. This fails to hold for $p=1$ but Fujii \cite{fujii1979cesaro} proved that $%
\sup\limits_{n}\left\vert \sigma _{n}^{1}\right\vert $ is bounded from the
dyadic Hardy space $H_{1}(\mathbb{I})$ to the space $L_{1}(\mathbb{I})$ (see
also Simon \cite{simon1985}). Fujii's theorem was extended by Weisz \cite%
{weisz1996cesaro}. In particular, Weisz \cite{weisz1996cesaro} proved that
the maximum operator is bounded from the Hardy space $H_{p}(\mathbb{I})$\ to
the space $L_{p}(\mathbb{I}),$ when $p>1/2$ . The essence of the condition $%
p>1/2$ was proved by the author \cite{GogiEJA}. If the $\{q_{k}\}$ is an
non-decreasing sequence then it can be proved that the following inequality
occurs (see Persson, Tephnadze, Wall \cite{PTW}),%
\begin{equation}
\sup\limits_{n}\left\vert t_{n}\right\vert \leq c\sup\limits_{n}\left\vert
\sigma _{n}\right\vert ,  \label{<}
\end{equation}%
where by $t_{n}$ is denoted Nörlund means of Walsh-Fourier series. From (\ref%
{<}) it follows that the maximum operator $\sup\limits_{n}\left\vert
t_{n}\right\vert $ is bounded from the Hardy space $H_{p}(\mathbb{I})$\ to
the space $L_{p}(\mathbb{I}),$ when $p>1/2$.

The situation is different when the sequence $\{q_{k}\}$ is decreasing. Let
us cite the following two cases:

\begin{itemize}
\item say $q_{k}=A_{k}^{\alpha -1},\alpha \in \left( 0,1\right) $, where 
\begin{equation*}
A_{0}^{\alpha }=1,\,\,A_{n}^{\alpha }=\frac{\left( \alpha +1\right) \cdots
\left( \alpha +n\right) }{n!}.
\end{equation*}%
then it is easy to see that $\{q_{k}\}$ is decreasing and at the same time
the operator $\sup\limits_{n}\left\vert \sigma _{n}^{\alpha }\right\vert
\,\,\,\,\left( 0<\alpha <1\right) $ is bounded from the Hardy space $H_{p}(%
\mathbb{I})$\ to the space $L_{p}(\mathbb{I}),$ when $p>1/\left( 1+\alpha
\right) $ (see Weisz \cite{weisz2001c});

\item Assume that $q_{k}=1/k$. Then the sequence is decreasing, but the
maximum operator is not bounded from the Hardy space $H_{p}(\mathbb{I})$\ to
the space $L_{p}(\mathbb{I})$ by any $p\in (0,1]$ (see \cite{GogiRendi}).
\end{itemize}

Therefore, Nörlund means with non-increasing weights can be divided into two
groups:

\begin{itemize}
\item Nörlund means with non-increasing weights, whose corresponding maximum
operator is bounded from the Hardy space $H_{p}(\mathbb{I})$\ to the space $%
L_{p}(\mathbb{I})$ for some $p\in (0,1]$;

\item Nörlund means with non-increasing weights that are not bounded from
the Hardy space $H_{p}(\mathbb{I})$\ to the space $L_{p}(\mathbb{I})$ by any 
$p\in (0,1]$.
\end{itemize}

The presented paper will be proved the necessary and sufficient conditions
in order maximal operator of Nörlund means with non-increasing weights to be
bounded from the Hardy space $H_{p}(\mathbb{I})$\ to the space $L_{p}(%
\mathbb{I})$. It also follows from the established theorem that the
boundedness of maximal operator of Nörlund means with non-increasing weights
from the Hardy space $H_{1}(\mathbb{I})$\ to the space $L_{1}(\mathbb{I})$
is equivalent to the type $\left( \infty ,\infty \right) $ .

\section{Walsh Functions}

We denote the set of non-negative integers by $\mathbb{N}$. By a dyadic
interval in $\mathbb{I}:=[0,1)$ we mean one of the form $I\left( l,k\right)
:=\left[ \frac{l-1}{2^{k}},\frac{l}{2^{k}}\right) $ for some $k\in \mathbb{N}
$, $0<l\leq 2^{k}$. Given $k\in \mathbb{N}$ and $x\in \lbrack 0,1),$ let $%
I_{k}(x)$ denote the dyadic interval of length $2^{-k}$ which contains the
point $x$. We also use the notation $I_{n}:=I_{n}\left( 0\right) \left( n\in 
\mathbb{N}\right) ,\overline{I}_{k}\left( x\right) :=\mathbb{I}\backslash
I_{k}\left( x\right) $. Let 
\begin{equation*}
x=\sum\limits_{n=0}^{\infty }x_{n}2^{-\left( n+1\right) }
\end{equation*}%
be the dyadic expansion of $x\in \mathbb{I}$, where $x_{n}=0$ or $1$ and if $%
x$ is a dyadic rational number we choose the expansion which terminates in $%
0^{\prime }$s.

For any given $n\in \mathbb{N}$ it is possible to write $n$ uniquely as%
\begin{equation*}
n=\sum\limits_{k=0}^{\infty }\varepsilon _{k}\left( n\right) 2^{k},
\end{equation*}%
where $\varepsilon _{k}\left( n\right) =0$ or $1$ for $k\in \mathbb{N}$.
This expression will be called the binary expansion of $n$ and the numbers $%
\varepsilon _{k}\left( n\right) $ will be called the binary coefficients of $%
n$. Let us introduce for $1\leq n\in \mathbb{N}$ the notation $\left\vert
n\right\vert :=\max \{j\in \mathbb{N}\mathbf{:}\varepsilon _{j}\left(
n\right) \neq 0\}$, that is $2^{\left\vert n\right\vert }\leq
n<2^{\left\vert n\right\vert +1}.$

Let us set the $n$th $\left( n\in \mathbb{N}\right) $ Walsh-Paley function
at point $x\in \mathbb{I}$ as:%
\begin{equation*}
w_{n}\left( x\right) =\left( -1\right) ^{\sum\limits_{j=0}^{\infty
}\varepsilon _{j}\left( n\right) x_{j}}.
\end{equation*}%
Let us denote the logical addition on $\mathbb{I}$ by $\dotplus $. That is,
for any $x,y\in \mathbb{I}$ 
\begin{equation*}
x\dotplus y:=\sum\limits_{n=0}^{\infty }\left\vert x_{n}-y_{n}\right\vert
2^{-\left( n+1\right) }.
\end{equation*}

The $n$th Walsh-Dirichlet kernel is defined by 
\begin{equation*}
D_{n}\left( x\right) =\sum\limits_{k=0}^{n-1}w_{k}\left( x\right) .
\end{equation*}

Recall that \cite{GES,SWSP} 
\begin{equation}
D_{2^{n}}\left( x\right) =2^{n}\mathbf{1}_{I_{n}}\left( x\right) ,
\label{Dir}
\end{equation}%
where $\mathbf{1}_{E}$ is the characteristic function of the set $E$.

As usual, denote by $L_{1}\left( \mathbb{I}\right) $ the set of measurable
functions defined on $\mathbb{I}$, for which 
\begin{equation*}
\left\Vert f\right\Vert _{1}:=\int\limits_{\mathbb{I}}\left\vert f\left(
t\right) \right\vert dt<\infty \text{.}
\end{equation*}%
Let $f\in L_{1}\left( \mathbb{I}\right) $. The partial sums of the
Walsh-Fourier series are defined as follows:

\begin{equation*}
S_{M}\left( f;x\right) :=\sum\limits_{i=0}^{M-1}\widehat{f}\left( i\right)
w_{i}\left( x\right) ,
\end{equation*}%
where the number 
\begin{equation*}
\widehat{f}\left( i\right) =\int\limits_{\mathbb{I}}f\left( t\right)
w_{i}\left( t\right) dt
\end{equation*}%
is said to be the $i$th Walsh-Fourier coefficient of the function\thinspace $%
f.$ Let us set $E_{n}\left( f; x\right) =S_{2^{n}}\left( f;x\right)$. The
maximal function is defined by 
\begin{equation*}
E^{\ast }\left( f;x\right) =\sup\limits_{n\in \mathbb{N}}E_{n}\left(
f;x\right) .
\end{equation*}

\section{Walsh-N\"orlund means}

Let us set $\{q_{k}:k\geq 0\}$ be a sequence of non-negative numbers. We
define the $n$th Nörlund mean of the Walsh-Fourier series by 
\begin{equation}
t_{n}(f;x):=\frac{1}{Q_{n}}\sum_{k=1}^{n}q_{n-k}S_{k}(f;x),
\label{Norlund-mean}
\end{equation}%
where $Q_{n}:=\sum_{k=0}^{n-1}q_{k}$ $(n\geq 1).$ It is always assumed that $%
q_{0}>0$ and $\lim_{n\rightarrow \infty }Q_{n}=\infty $. In this case, the
summability method generated by the sequence $\{q_{k}:k\geq 0\}$ is regular
(see \cite{Moricz-Siddiqi}) if and only if 
\begin{equation}
\lim_{n\rightarrow \infty }\frac{q_{n-1}}{Q_{n}}=0.  \label{reg2}
\end{equation}%
The Nörlund kernels are defined by 
\begin{equation*}
{F}_{n}(t):=\frac{1}{Q_{n}}\sum_{k=1}^{n}q_{n-k}D_{k}(t).
\end{equation*}%
The Fejér means and kernels are 
\begin{equation*}
\sigma _{n}(f,x):=\frac{1}{n}\sum_{k=1}^{n}S_{k}(f,x),\quad K_{n}(t):=\frac{1%
}{n}\sum_{k=1}^{n}D_{k}(t).
\end{equation*}%
It is easily seen that the means $t_{n}(f)$ and $\sigma _{n}(f)$ can be got
by convolution of $f$ with the kernels $F_{n}$ and $K_{n}$. That is, 
\begin{eqnarray*}
t_{n}(f,x) &=&\int_{G}f(x\dotplus t){\ F}_{n}(t)dt=\left( f\ast F_{n}\right)
\left( x\right) ,\quad \\
\sigma _{n}(f,x) &=&\int_{G}f(x\dotplus t)K_{n}(t)dt=\left( f\ast
K_{n}\right) \left( x\right) .
\end{eqnarray*}%
It is well-known that $L_{1}$ norm of Fejér kernels are uniformly bounded,
that is 
\begin{equation}
\Vert K_{n}\Vert _{1}\leq c\text{ for all }n\in \mathbb{N}\text{.}
\label{Fejer}
\end{equation}%
Yano estimated the value of $c$ and he have $c=2$ \cite{Yano1951}. Recently,
in paper (see \cite{Toledo-18}) it was shown that the exact value of $c$ is $%
\frac{17}{15}$.

\section{Auxiliary Propositions}

In order to prove main results we need the following theorems.

\begin{SWS}
\label{SWS} Let $n\in \mathbb{N}$ and $e_j:=2^{-j-1}$. Then 
\begin{equation}
K_{2^{n}}\left( x\right) =\frac{1}{2}\left( 2^{-n}D_{2^{n}}\left( x\right)
+\sum\limits_{j=0}^{n}2^{j-n}D_{2^{n}}\left( x\dotplus e_j \right) \right) .
\label{SW}
\end{equation}
\end{SWS}

The proof can be found in \cite{SWSP}.

\begin{GN1}
\label{lemma-decomp} Let $n=2^{n_{1}}+2^{n_{2}}+\cdots +2^{n_{r}}$ with $%
n_{1}>n_{2}>\cdots >n_{r}\geq 0$. Let us set $n^{(0)}:=n$ and $%
n^{(i)}:=n^{(i-1)}-2^{n_{i}}$ ($i=1.\ldots ,r-1$), $n^{(r)}:=0$. Then the
following decomposition holds. 
\begin{eqnarray}
{F}_{n} &=&\frac{w_{n}}{Q_{n}}%
\sum_{j=1}^{r}Q_{n^{(j-1)}}w_{2^{n_{j}}}D_{2^{n_{j}}}  \label{Fn1+Fn2} \\
&&-\frac{w_{n}}{Q_{n}}\sum_{j=1}^{r}w_{n^{(j-1)}}w_{2^{n_{j}}-1}%
\sum_{k=1}^{2^{n_{j}}-1}q_{k+n^{(j)}}D_{k}  \notag \\
&=&:F_{n,1}+F_{n,2}.  \notag
\end{eqnarray}
\end{GN1}

\begin{GN2}
\label{thm-norm} Let $\{q_k : k\in {\mathbb{N}}\}$ be a sequence of
non-negative numbers. If the sequence $\{q_{k}:k\in {\mathbb{N}}\}$ is
monotone non-increasing (in sign $q_{k}\downarrow $). Then 
\begin{equation}
\Vert F_{n}\Vert _{1}\sim \frac{1}{Q_{n}}\sum_{k=1}^{|n|}|\varepsilon
_{k}(n)-\varepsilon _{k+1}(n)|Q_{2^{k}}.  \label{norm-norlund-2}
\end{equation}
\end{GN2}

The proof of Theorems GN1 and GN2 can be found in \cite{goginava2022some}.

Using Abel's transformation we have 
\begin{eqnarray*}
\sum_{k=1}^{2^{n_{j}}-1}q_{k+n^{(j)}}D_{k} &=&\sum_{k=1}^{2^{n_{j}}-2}\left(
q_{k+n^{(j)}}-q_{k+n^{(j)}+1}\right) kK_{k} \\
&&+q_{n^{(j-1)}-1}{(2^{n_{j}}-1})K_{2^{n_{j}}-1}.
\end{eqnarray*}%
Thus, we get%
\begin{eqnarray}
F_{n,2} &=&\frac{w_{n}}{Q_{n}}\sum_{j=1}^{r}%
\sum_{k=1}^{2^{n_{j}}-2}w_{n^{(j-1)}}w_{2^{n_{j}}-1}\left(
q_{k+n^{(j)}}-q_{k+n^{(j)}+1}\right) kK_{k}  \label{F1+F2} \\
&&+\frac{w_{n}}{Q_{n}}%
\sum_{j=1}^{r}w_{n^{(j-1)}}w_{2^{n_{j}}-1}q_{n^{(j-1)}-1}{(2^{n_{j}}-1}%
)K_{2^{n_{j}}-1}  \notag \\
&=&:F_{n,2}^{\left( 1\right) }+F_{n,2}^{\left( 2\right) }.  \notag
\end{eqnarray}

\begin{lemma}
\label{gogi} \label{m(p>1/2)}Let $p\in \left( \frac{1}{2},1\right] $. Then%
\begin{equation*}
\int\limits_{\mathbb{I}}\sup\limits_{1\leq n\leq 2^{N}}\left( n\left\vert
K_{n}\right\vert \right) ^{p}\leq c_{p}2^{N\left( 2p-1\right) }.
\end{equation*}
\end{lemma}

\begin{proof}[Proof of Lemma \protect\ref{gogi}]
Let $p=1$. Since (see \cite{SWSP}) 
\begin{equation}
n\left\vert K_{n}\left( x\right) \right\vert \leq
c\sum\limits_{s=0}^{|n|}2^{s}K_{2^{s}}\left( x\right)  \label{fest}
\end{equation}%
from (\ref{Fejer}) we have \newline
\begin{equation*}
\int\limits_{\mathbb{I}}\sup\limits_{1\leq n\leq 2^{N}}\left( n\left\vert
K_{n}\left( x\right) \right\vert \right) dx\leq
c\sum\limits_{s=0}^{N}2^{s}\int\limits_{\mathbb{I}}K_{2^{s}}\left( x\right)
dx\leq c2^{N}.
\end{equation*}

Let \thinspace $1/2<p<1$. Applying the inequality 
\begin{equation*}
\left( \sum\limits_{k=0}^{\infty }a_{k}\right) ^{p}\leq
\sum\limits_{k=0}^{\infty }a_{k}^{p}\,\,\,\,\,\,\,\,\,\,\left( a_{k}\geq
0,\,\,\,\,0<p\leq 1\right)
\end{equation*}
and (see \cite{GogiJMAA}) 
\begin{equation*}
\int\limits_{\mathbb{I}}\left( 2^{s}K_{2^{s}}\left( x\right) \right)
^{p}dx\leq c_{p}2^{s\left( 2p-1\right) },1/2<p<1
\end{equation*}
we get 
\begin{equation*}
\int\limits_{\mathbb{I}}\sup\limits_{1\leq n\leq 2^{N}}\left( n\left\vert
K_{n}\left( x\right) \right\vert \right) ^{p}dx\leq
c_{p}\sum\limits_{s=0}^{N}\int\limits_{\mathbb{I}}\left(
2^{s}K_{2^{s}}\left( x\right) \right) ^{p}dx\leq c_{p}2^{N\left( 2p-1\right)
}.
\end{equation*}

Lemma \ref{gogi} is proved.
\end{proof}

\section{Dyadic Hardy Spaces}

The norm (or quasinorm) of the space $L_{p}\left( \mathbb{I}\right) $ is
defined by 
\begin{equation*}
\left\Vert f\right\Vert _{p}:=\left( \int\limits_{\mathbb{I}}\left\vert
f\left( x\right) \right\vert ^{p}dx\right) ^{1/p}\,\,\,\,\left( 0<p<+\infty
\right) .
\end{equation*}

In case $p=\infty$, by $L_{p}(\mathbb{I})$ we mean $L_{\infty}(\mathbb{I})$,
endoved with the supremum norm.

The space weak-$L_{1}\left( \mathbb{I}\right) $ consists of all measurable
functions $f$ for which 
\begin{equation*}
\left\Vert f\right\Vert _{\text{weak}-L_{1}\left( \mathbb{I}\right)
}:=\sup\limits_{\lambda >0}\lambda \left\vert \left( \left\vert f\right\vert
>\lambda \right) \right\vert <+\infty .
\end{equation*}

Let $f\in L_{1}\left( \mathbb{I}\right) $. \ For $0<p<\infty $ the Hardy
space $H_{p}(\mathbb{I})$ consists all functions for which

\begin{equation*}
\left\Vert f\right\Vert _{H_{p}}:=\left\Vert E^{\ast }\left( f\right)
\right\Vert _{p}<\infty .
\end{equation*}

A bounded measurable function $a$ is a p-atom, if there exists a
dyadic\thinspace in\-ter\-val $I$, such that

a) $\int\limits_{I}a=0$;

b) $\left\Vert a\right\Vert _{\infty }\leq \left\vert I\right\vert ^{-1/p}$;

c) supp $a\subset I$.

An operator $T$ be called p-quasi-local if there exist a constant $c_{p}>0$
such that for every p-atom $a$

\begin{equation*}
\int\limits_{\mathbb{I}\backslash I}|Ta|^{p}\leq c_{p}<\infty ,
\end{equation*}%
where $I$ is the support of the atom. We shall need the following

\begin{W1}
\label{Weisz} Suppose that the operator $T$ \thinspace is\thinspace $\sigma $%
-sublinear and $p$-quasi-local for each $0<p\leq 1$. If $T$ is bounded from $%
L_{\infty }(\mathbb{I})$ to $L_{\infty }(\mathbb{I})$, then

\begin{equation*}
\left\Vert Tf\right\Vert _{p}\leq c_{p}\left\Vert f\right\Vert
_{p}\,\,\,\,\,\,\,\,\,\,\,(f\in H_{p}\left( \mathbb{I}\right) )
\end{equation*}%
for every $0<p<\infty $ $.\,$In particular for $f\in L_{1}(\mathbb{I})$, it
holds

\begin{equation*}
\left\Vert Tf\right\Vert _{weak\_L_{1}(\mathbb{I})}\leq C\left\Vert
f\right\Vert _{1}.
\end{equation*}
\end{W1}

\begin{W2}
\label{interpolation} If a sublinear operator is bounded from $H_{p_{0}}(%
\mathbb{I})$ to $L_{p_{0}(\mathbb{I})}$ and from $L_{p_{1}}(\mathbb{I})$ to $%
L_{p_{1}}\left( \mathbb{I}\right) \left( p_{0}\leq 1<p_{1}\leq \infty
\right) $ then it is also bounded from $H_{p}(\mathbb{I})$ to $L_{p}(\mathbb{%
I})$ if $p_{0}<p<p_{1}$.
\end{W2}

The proofs of Theorems W1 and W2 can be found in \cite{Weisz-book1}.

\section{\protect\bigskip Maximal Operators of Walsh-Nörlund means}

The goal of this paragraph is to investigate the boundedness of the maximal
operators of the Walsh-Nörlund means on the dyadic Hardy spaces. More
precisely, to find the necessary and sufficient conditions for the maximal
operator of the Walsh-Nörlund means to be bounded from the Hardy space $%
H_{p}(\mathbb{I})$ to the space $L_{p}(\mathbb{I})$ for fixed $p\in (0,1]$.

Let us first prove that if the condition 
\begin{equation}
\sup\limits_{n\in \mathbb{N}}\frac{1}{Q_{n}}\sum_{k=1}^{|n|}|\varepsilon
_{k}(n)-\varepsilon _{k+1}(n)|Q_{2^{k}}=\infty  \label{infinity}
\end{equation}
is fulfilled, then the boundedness of the maximum operator from the Hardy
space $H_{1}(\mathbb{I})$ to the space $L_{1}(\mathbb{I})$ does not occur.
Moreover, we prove that the following is valid

\begin{theorem}
\label{nobound}Let $\left\{ m_{A}:A\in \mathbb{N}\right\} $ be a subsequence
for which the condition%
\begin{equation*}
\sup\limits_{A\in \mathbb{N}}\frac{1}{Q_{m_{A}}}\sum_{k=1}^{|m_{A}|}|%
\varepsilon _{k}(m_{A})-\varepsilon _{k+1}(m_{A})|Q_{2^{k}}=\infty
\end{equation*}%
holds.$\ $\ The operator $t_{m_{A}}\left( f\right) $ is not uniformly
bounded from the dyadic Hardy spaces $H_{1}\left( \mathbb{I}\right) $ to the
space $L_{1}\left( \mathbb{I}\right) $.
\end{theorem}

\begin{proof}[Proof of Theorem \protect\ref{nobound}]
Set%
\begin{equation*}
f_{A}:=D_{2^{\left\vert m_{A}\right\vert +1}}-D_{2^{\left\vert
m_{A}\right\vert }}.
\end{equation*}%
Then it is easy to see that%
\begin{equation*}
\sup\limits_{n\in \mathbb{N}}\left\vert S_{2^{n}}\left( f_{A}\right)
\right\vert =D_{2^{\left\vert m_{A}\right\vert }}
\end{equation*}%
and consequently,%
\begin{equation*}
\left\Vert f_{A}\right\Vert _{H_{1}}=\left\Vert \sup\limits_{n\in \mathbb{N}%
}\left\vert S_{2^{n}}\left( f_{A}\right) \right\vert \right\Vert
_{1}=\left\Vert D_{2^{\left\vert m_{A}\right\vert }}\right\Vert _{1}=1.
\end{equation*}%
Set 
\begin{equation*}
m_{A}=2^{\left\vert m_{A}\right\vert }+q_{A},
\end{equation*}%
where 
\begin{equation*}
q_{A}:=\sum\limits_{j=0}^{\left\vert m_{A}\right\vert -1}\varepsilon
_{j}\left( m_{A}\right) 2^{j}.
\end{equation*}%
Then we can write%
\begin{equation*}
t_{m_{A}}\left( f_{A}\right) =\frac{1}{Q_{m_{A}}}\sum\limits_{k=2^{\left%
\vert m_{A}\right\vert }+1}^{2^{\left\vert m_{A}\right\vert
}+q_{A}-1}q_{m_{A}-k}S_{k}\left( f_{A}\right) .
\end{equation*}%
It is easy to see that%
\begin{eqnarray*}
S_{k}\left( f_{A}\right) &=&S_{k}\left( D_{2^{\left\vert m_{A}\right\vert
+1}}-D_{2^{\left\vert m_{A}\right\vert }}\right) \\
&=&S_{2^{\left\vert m_{A}\right\vert +1}}\left( D_{k}\right) -S_{k}\left(
D_{2^{\left\vert m_{A}\right\vert }}\right) \\
&=&D_{k}-D_{2^{\left\vert m_{A}\right\vert }},2^{|m_{A}|}<k\leq m_{A}\text{.}
\end{eqnarray*}%
Hence, we have%
\begin{eqnarray*}
t_{m_{A}}\left( f_{A}\right) &=&\frac{1}{Q_{m_{A}}}\sum\limits_{k=2^{\left%
\vert m_{A}\right\vert }+1}^{2^{\left\vert m_{A}\right\vert
}+q_{A}}q_{m_{A}-k}\left( D_{k}-D_{2^{\left\vert m_{A}\right\vert }}\right)
\\
&=&\frac{1}{Q_{m_{A}}}\sum\limits_{k=1}^{q_{A}}q_{q_{A}-k}\left(
D_{k+2^{\left\vert m_{A}\right\vert }}-D_{2^{\left\vert m_{A}\right\vert
}}\right) \\
&=&\frac{w_{2^{\left\vert m_{A}\right\vert }}}{Q_{m_{A}}}\sum%
\limits_{k=1}^{q_{A}}q_{q_{A}-k}D_{k}.
\end{eqnarray*}%
From the condition of Theorem \ref{nobound} and by (\ref{norm-norlund-2}) we
conclude that%
\begin{equation*}
\sup\limits_{A\in \mathbb{N}}\left\Vert t_{m_{A}}\left( f_{A}\right)
\right\Vert _{1}=\sup\limits_{A\in \mathbb{N}}\left\Vert
F_{q_{A}}\right\Vert _{1}=\infty .
\end{equation*}%
Theorem \ref{nobound} is proved.
\end{proof}

Now, we prove that the maximal operator of Walsh-Nörlund means with
non-increasing weights can not be bounded from the Hardy space $%
H_{1/2}\left( \mathbb{I}\right) $ to the space $L_{1/2}\left( \mathbb{I}%
\right) $. Based on the interpolation Theorem W2, the maximum operator of
Walsh-Nörlund means with non-increasing weights can not be bounded from the
Hardy space $H_{p}\left( \mathbb{I}\right) $ to the space $L_{p}\left( 
\mathbb{I}\right) $ when $p<1/2$ (see Persson, Tephnadze, Wall \cite{PTW}).

\begin{theorem}
\label{p<1/2} The maximal operator of Walsh-Nörlund means with
non-increasing weights can not be bounded from the Hardy space $%
H_{1/2}\left( \mathbb{I}\right) $ to the space $L_{1/2}\left( \mathbb{I}%
\right) $.
\end{theorem}

\begin{proof}[Proof of Theorem \protect\ref{p<1/2}]
Set%
\begin{equation*}
f_{n}:=D_{2^{n+1}}-D_{2^{n}}.
\end{equation*}%
Then it is easy to see that%
\begin{equation*}
\sup\limits_{m\in \mathbb{N}}\left\vert S_{2^{m}}\left( f_{n}\right)
\right\vert =D_{2^{n}}
\end{equation*}%
and consequently,%
\begin{equation}
\left\Vert f_{n}\right\Vert _{H_{p}}=\left\Vert \sup\limits_{m\in \mathbb{N}%
}\left\vert S_{2^{m}}\left( f_{n}\right) \right\vert \right\Vert
_{p}=\left\Vert D_{2^{n}}\right\Vert _{p}=2^{n\left( 1-1/p\right) }.
\label{f}
\end{equation}%
Let $s<n$. Then we can write 
\begin{eqnarray*}
t_{2^{n}+2^{s}}\left( f_{n}\right) &=&\frac{1}{Q_{2^{n}+2^{s}}}%
\sum\limits_{j=1}^{2^{n}+2^{s}}q_{2^{n}+2^{s}-j}S_{j}\left( f_{n}\right) \\
&=&\frac{1}{Q_{2^{n}+2^{s}}}\sum%
\limits_{j=2^{n}+1}^{2^{n}+2^{s}}q_{2^{n}+2^{s}-j}S_{j}\left(
D_{2^{n+1}}-D_{2^{n}}\right) \\
&=&\frac{1}{Q_{2^{n}+2^{s}}}\sum%
\limits_{j=2^{n}+1}^{2^{n}+2^{s}}q_{2^{n}+2^{s}-j}\left( S_{2^{n+1}}\left(
D_{j}\right) -S_{j}\left( D_{2^{n}}\right) \right) \\
&=&\frac{1}{Q_{2^{n}+2^{s}}}\sum%
\limits_{j=2^{n}+1}^{2^{n}+2^{s}}q_{2^{n}+2^{s}-j}\left(
D_{j}-D_{2^{n}}\right) \\
&=&\frac{1}{Q_{2^{n}+2^{s}}}\sum\limits_{j=1}^{2^{s}}q_{2^{s}-j}\left(
D_{j+2^{n}}-D_{2^{n}}\right) \\
&=&\frac{w_{2^{n}}}{Q_{2^{n}+2^{s}}}\sum%
\limits_{j=1}^{2^{s}}q_{2^{s}-j}D_{j}.
\end{eqnarray*}%
Consequently,%
\begin{eqnarray}
&&\int\limits_{\mathbb{I}}\left( \sup\limits_{1\leq s<n}\left\vert
t_{2^{n}+2^{s}}\left( f_{n}\right) \right\vert \right) ^{p}  \label{Q} \\
&\geq &\sum\limits_{s=0}^{n-1}\int\limits_{I_{s}\backslash
I_{s+1}}\left\vert t_{2^{n}+2^{s}}\left( f_{n}\right) \right\vert ^{p} 
\notag \\
&=&\sum\limits_{s=0}^{n-1}\int\limits_{I_{s}\backslash I_{s+1}}\frac{1}{%
Q_{2^{n}+2^{s}}^{p}}\left\vert
\sum\limits_{j=1}^{2^{s}}q_{2^{s}-j}D_{j}\right\vert ^{p}  \notag \\
&=&\sum\limits_{s=0}^{n-1}\frac{1}{2^{s+1}Q_{2^{n}+2^{s}}^{p}}\left\vert
\sum\limits_{j=1}^{2^{s}}q_{2^{s}-j}j\right\vert ^{p}.  \notag
\end{eqnarray}

Since $q_{k}$ is non-increasing we can write%
\begin{eqnarray*}
Q_{2^{n}}
&=&\sum\limits_{k=0}^{2^{n-1}-1}q_{k}+\sum\limits_{k=2^{n-1}}^{2^{n}-1}q_{k}
\\
&\leq &2\sum\limits_{k=0}^{2^{n-1}-1}q_{k}=2Q_{2^{n-1}} \\
&\leq &\cdots \leq 2^{n-s}Q_{2^{s}}
\end{eqnarray*}%
and%
\begin{equation}
\frac{Q_{2^{s}}}{2^{s}}\geq \frac{Q_{2^{n}}}{2^{n}}\left( s\leq n\right) .
\label{Q1}
\end{equation}%
Combine (\ref{Q}) and (\ref{Q1}) we get $\left( p=1/2\right) $%
\begin{eqnarray*}
&&\int\limits_{\mathbb{I}}\left( \sup\limits_{1\leq s<n}\left\vert
t_{2^{n}+2^{s}}\left( f_{n}\right) \right\vert \right) ^{1/2} \\
&\geq &\sum\limits_{s=0}^{n-1}\frac{1}{2^{s+1}Q_{2^{n}+2^{s}}^{1/2}}%
\left\vert \sum\limits_{j=2^{s-1}+1}^{2^{s}}q_{2^{s}-j}j\right\vert ^{1/2} \\
&\geq &c\sum\limits_{s=0}^{n-1}\frac{2^{s/2}Q_{2^{s}}^{1/2}}{%
2^{s}Q_{2^{n}}^{1/2}} \\
&\geq &c\sum\limits_{s=0}^{n-1}\frac{1}{2^{s/2}}\left( \frac{2^{s}}{2^{n}}%
\right) ^{1/2} \\
&=&\frac{cn}{2^{n/2}}.
\end{eqnarray*}%
Hence,%
\begin{equation*}
\frac{\left\Vert t^{\ast }\left( f_{n}\right) \right\Vert _{_{1/2}}^{1/2}}{%
\left\Vert f_{n}\right\Vert _{H_{1/2}}^{1/2}}\geq cn\rightarrow \infty
\end{equation*}%
as $n\rightarrow \infty $ $.$ Theorem \ref{p<1/2} is proved.
\end{proof}

Finally, we are ready to formulate a basic problem: \textit{to say }$\left\{
q_{k}\right\} $\textit{\ is a non-increasing and positive sequence. It is
known that the operator }$t^{\ast }\left( f\right) $\textit{\ is bounded
from }$L_{\infty }(\mathbb{I})$\textit{\ to }$L_{\infty }(\mathbb{I})$%
\textit{\ and }$p\in (1/2,1]$\textit{. Find the necessary and sufficient
conditions for the }$\left\{ q_{k}\right\} $\textit{\ sequence in order for
the maximum operator $t^{\ast }\left( f\right) $ to be bounded from the
Hardy space }$H_{p}(\mathbb{I})$\textit{\ to the space }$L_{p}(\mathbb{I})$%
\textit{.}

The paper will provide a complete answer to the given question. The
following theorem is true.

\begin{theorem}
\label{main} Let $\left\{ q_{k}\right\} $\textit{\ be a non-increasing and
positive sequence. It is known that the operator }$t^{\ast }\left( f\right) $%
\textit{\ is bounded from }$L_{\infty }(\mathbb{I})$\textit{\ to }$L_{\infty
}(\mathbb{I})$\textit{\ and }$p\in (1/2,1]$\textit{. }In order for the given
operator to be \textit{bounded from the Hardy space }$H_{p}(\mathbb{I})$%
\textit{\ to the space }$L_{p}(\mathbb{I})$ it is necessary and sufficient
that%
\begin{equation*}
\sup\limits_{N}\frac{2^{N\left( 1-p\right) }}{Q_{2^{N}}^{p}}%
\sum\limits_{j=1}^{N}Q_{2^{j}}^{p}2^{j\left( p-1\right) }<\infty .
\end{equation*}
\end{theorem}

\begin{proof}[Proof of Theorem \protect\ref{main}]
\textit{\textbf{Necessity}. }We asuume that%
\begin{equation}
\sup\limits_{N}\frac{2^{N\left( 1-p\right) }}{Q_{2^{N}}^{p}}%
\sum\limits_{j=1}^{N}Q_{2^{j}}^{p}2^{j\left( p-1\right) }=\infty .
\label{nb}
\end{equation}%
From (\ref{f}) and (\ref{Q}) we have%
\begin{eqnarray*}
&&\frac{\left\Vert t^{\ast }\left( f_{n}\right) \right\Vert _{_{p}}^{p}}{%
\left\Vert f_{n}\right\Vert _{H_{p}}^{p}} \\
&\geq &c_{p}2^{n\left( 1-p\right) }\int\limits_{\mathbb{I}}\left(
\sup\limits_{1\leq s<n}\left\vert t_{2^{n}+2^{s}}\left( f_{n}\right)
\right\vert \right) ^{p} \\
&\geq &c_{p}2^{n\left( 1-p\right) }\sum\limits_{s=0}^{n-1}\frac{1}{%
2^{s+1}Q_{2^{n}+2^{s}}^{p}}\left\vert
\sum\limits_{j=1}^{2^{s}}q_{2^{s}-j}j\right\vert ^{p} \\
&\geq &c_{p}\frac{2^{n\left( 1-p\right) }}{Q_{2^{n}}^{p}}\sum%
\limits_{s=0}^{n-1}\frac{1}{2^{s}}\left\vert
\sum\limits_{j=2^{s-1}+1}^{2^{s}}q_{2^{s}-j}j\right\vert ^{p} \\
&\geq &c_{p}\frac{2^{n\left( 1-p\right) }}{Q_{2^{n}}^{p}}\sum%
\limits_{s=0}^{n-1}2^{s\left( p-1\right) }Q_{2^{s}}^{p}.
\end{eqnarray*}%
Then from (\ref{nb}) we get%
\begin{equation*}
\sup\limits_{n\in \mathbb{N}}\frac{\left\Vert t^{\ast }\left( f_{n}\right)
\right\Vert _{_{p}}^{p}}{\left\Vert f_{n}\right\Vert _{H_{p}}^{p}}=\infty
\end{equation*}%
and consequently, the operator $t^{\ast }$ is not bounded \textit{from the
Hardy space }$H_{p}\left( \mathbb{I}\right) $\textit{\ to the space }$%
L_{p}\left( \mathbb{I}\right) .$

\textit{\textbf{Sufficiency.}} We suppose that $f\in H_{p}\left( \mathbb{I}%
\right) $. Let function $a$ be an $H_{p}$ atom. It means that either $a$ is
constant or there is an interval $I_{N}(u)$ such that supp$\left( a\right)
\subset I_{N}(u)$, $\Vert a\Vert _{\infty }\leq 2^{N/p}$ and $\int a=0$.
Without lost of generality we can suppose that $u=0$. Consequently, for any
function $g$ which is $\mathcal{A}_{N}$-measurable we have that $\int ag=0$.
We prove that the operator $\sup\limits_{n>N}\left( f\ast F_{n}\right)
\left( x\right) $ is $H_{p}$-quasi local. That is,%
\begin{equation}
\int\limits_{\overline{I}_{N}}\left( \sup\limits_{n>N}\left\vert a\ast
F_{n}\right\vert \right) ^{p}\leq c_{p}\text{.}  \label{QL}
\end{equation}%
Let $x\in \overline{I}_{N}$. Then from (\ref{Fn1+Fn2}) we can write%
\begin{eqnarray}
\left\vert a\ast F_{n}\right\vert &=&\left\vert \int\limits_{\mathbb{I}%
}a\left( t\right) F_{n}\left( x\dotplus t\right) dt\right\vert \leq
2^{N/p}\int\limits_{I_{N}}\left\vert F_{n}\left( x\dotplus t\right)
\right\vert dt  \label{!+2} \\
&=&2^{N/p}\int\limits_{I_{N}}\left\vert F_{n,1}\left( x\dotplus t\right)
\right\vert dt+2^{N/p}\int\limits_{I_{N}}\left\vert F_{n,2}\left( x\dotplus
t\right) \right\vert dt.  \notag
\end{eqnarray}%
We have%
\begin{eqnarray*}
&&\int\limits_{I_{N}}\left\vert F_{n,1}\left( x\dotplus t\right) \right\vert
dt \\
&\leq &\frac{1}{Q_{n}}\sum_{j=1,n_{j}>N}^{r}Q_{n^{(j-1)}}\int%
\limits_{I_{N}}D_{2^{n_{j}}}\left( x\dotplus t\right) dt \\
&&+\frac{1}{Q_{n}}\sum_{j=1,n_{j}\leq
N}^{r}Q_{n^{(j-1)}}\int\limits_{I_{N}}D_{2^{n_{j}}}\left( x\dotplus t\right)
dt.
\end{eqnarray*}%
Since $t\in I_{N}$ and $x\notin I_{N}$ we have that $x\dotplus t\notin I_{N}$
and consequently by (\ref{Dir}) we get $D_{2^{n_{j}}}\left( x\dotplus
t\right) =0$ for $n_{j}>N.$ On the other hand, $\int%
\limits_{I_{N}}D_{2^{n_{j}}}\left( x\dotplus t\right) dt=\frac{1}{2^{N}}%
D_{2^{n_{j}}}\left( x\right) $ for $n_{j}\leq N$. Hence, we obtain%
\begin{eqnarray*}
&&\int\limits_{I_{N}}\left\vert F_{n,1}\left( x\dotplus t\right) \right\vert
dt \\
&\leq &\frac{1}{Q_{n}}\sum_{j=1,n_{j}\leq
N}^{r}Q_{n^{(j-1)}}\int\limits_{I_{N}}D_{2^{n_{j}}}\left( x\dotplus t\right)
dt \\
&=&\frac{1}{2^{N}Q_{n}}\sum_{j=1,n_{j}\leq
N}^{r}Q_{n^{(j-1)}}D_{2^{n_{j}}}\left( x\right) \\
&\leq &\frac{1}{2^{N}Q_{2^{N}}}\sum_{j=1}^{N}Q_{2^{j}}D_{2^{j}}\left(
x\right) .
\end{eqnarray*}%
Consequently, from the condition of the theorem we get%
\begin{eqnarray}
&&\int\limits_{\overline{I}_{N}}\sup\limits_{n>N}\left(
2^{N/p}\int\limits_{I_{N}}\left\vert F_{n,1}\left( x\dotplus t\right)
\right\vert dt\right) ^{p}dx  \label{1-2} \\
&\leq &\frac{c_{p}2^{N}}{2^{Np}Q_{2^{N}}^{p}}\sum_{j=1}^{N}Q_{2^{j}}^{p}\int%
\limits_{\overline{I}_{N}}D_{2^{j}}^{p}\left( x\right) dx  \notag \\
&=&\frac{c_{p}2^{N\left( 1-p\right) }}{Q_{2^{N}}^{p}}%
\sum_{j=1}^{N}Q_{2^{j}}^{p}2^{j\left( p-1\right) }\leq c_{p}<\infty .  \notag
\end{eqnarray}

From (\ref{F1+F2}) we get%
\begin{eqnarray}
&&\int\limits_{I_{N}}\left\vert F_{n,2}\left( x\dotplus t\right) \right\vert
dt  \label{int(F1+F2)} \\
&\leq &\int\limits_{I_{N}}\left\vert F_{n,2}^{\left( 1\right) }\left(
x\dotplus t\right) \right\vert dt+\int\limits_{I_{N}}\left\vert
F_{n,2}^{\left( 2\right) }\left( x\dotplus t\right) \right\vert dt.  \notag
\end{eqnarray}

We can write%
\begin{eqnarray*}
&&\frac{1}{Q_{n}}\sum_{j=1}^{r}\sum_{k=1}^{2^{n_{j}}-1}\left(
q_{k+n^{(j)}}-q_{k+n^{(j)}+1}\right) k\left\vert K_{k}\right\vert \\
&=&\frac{1}{Q_{n}}\sum_{j=1}^{r}\sum\limits_{m=1}^{n_{j}}%
\sum_{k=2^{m-1}}^{2^{m}-1}\left( q_{k+n^{(j)}}-q_{k+n^{(j)}+1}\right)
k\left\vert K_{k}\right\vert \\
&=&\frac{1}{Q_{n}}\sum_{j=1}^{r}\sum\limits_{m=1}^{n_{j+1}}%
\sum_{k=2^{m-1}}^{2^{m}-1}\left( q_{k+n^{(j)}}-q_{k+n^{(j)}+1}\right)
k\left\vert K_{k}\right\vert \\
&&+\frac{1}{Q_{n}}\sum_{j=1}^{r}\sum\limits_{m=n_{j+1}+1}^{n_{j}}%
\sum_{k=2^{m-1}}^{2^{m}-1}\left( q_{k+n^{(j)}}-q_{k+n^{(j)}+1}\right)
k\left\vert K_{k}\right\vert \\
&\leq &\frac{1}{Q_{n}}\sum_{j=1}^{r}\sum\limits_{m=1}^{n_{j+1}}\sup%
\limits_{2^{m-1}\leq k<2^{m}}\left( k\left\vert K_{k}\right\vert \right)
\sum_{k=2^{m-1}}^{2^{m}-1}\left( q_{k+n^{(j)}}-q_{k+n^{(j)}+1}\right) \\
&&+\frac{1}{Q_{n}}\sum_{j=1}^{r}\sum\limits_{m=n_{j+1}+1}^{n_{j}}\sup%
\limits_{2^{m-1}\leq k<2^{m}}\left( k\left\vert K_{k}\right\vert \right)
\sum_{k=2^{m-1}}^{2^{m}-1}\left( q_{k+n^{(j)}}-q_{k+n^{(j)}+1}\right) \\
&=&\frac{1}{Q_{n}}\sum_{j=1}^{r}\sum\limits_{m=1}^{n_{j+1}}\sup%
\limits_{2^{m-1}\leq k<2^{m}}\left( k\left\vert K_{k}\right\vert \right)
\left( q_{2^{m-1}+n^{(j)}}-q_{2^{m}+n^{(j)}}\right)
\end{eqnarray*}%
\begin{eqnarray*}
&&+\frac{1}{Q_{n}}\sum_{j=1}^{r}\sum\limits_{m=n_{j+1}+1}^{n_{j}}\sup%
\limits_{2^{m-1}\leq k<2^{m}}\left( k\left\vert K_{k}\right\vert \right)
\left( q_{2^{m-1}+n^{(j)}}-q_{2^{m}+n^{(j)}}\right) \\
&\leq &\frac{1}{Q_{n}}\sum_{j=1}^{r}q_{2^{n_{j+1}}}\sum%
\limits_{m=1}^{n_{j+1}}\sup\limits_{2^{m-1}\leq k<2^{m}}\left( k\left\vert
K_{k}\right\vert \right) \\
&&+\frac{1}{Q_{n}}\sum_{j=1}^{r}\sum\limits_{m=n_{j+1}+1}^{n_{j}}q_{2^{m-1}}%
\sup\limits_{2^{m-1}\leq k<2^{m}}\left( k\left\vert K_{k}\right\vert \right)
\\
&\leq &\frac{1}{Q_{n}}\sum_{j=1}^{n_{1}}q_{2^{j}}\sum\limits_{m=1}^{j}\sup%
\limits_{2^{m-1}\leq k<2^{m}}\left( k\left\vert K_{k}\right\vert \right) \\
&&+\frac{1}{Q_{n}}\sum\limits_{m=1}^{n_{1}}q_{2^{m-1}}\sup\limits_{2^{m-1}%
\leq k<2^{m}}\left( k\left\vert K_{k}\right\vert \right) \\
&\leq &\frac{2}{Q_{n}}\sum_{j=1}^{n_{1}}q_{2^{j-1}}\sum\limits_{m=1}^{j}\sup%
\limits_{2^{m-1}\leq k<2^{m}}\left( k\left\vert K_{k}\right\vert \right) .
\end{eqnarray*}%
Consequently, we have%
\begin{eqnarray*}
&&\int\limits_{I_{N}}\left\vert F_{n,2}^{\left( 1\right) }\left( x\dotplus
t\right) \right\vert dt \\
&\leq &\frac{2}{2^{N}Q_{n}}\sum_{j=1}^{N}q_{2^{j-1}}\sum\limits_{m=1}^{j}%
\sup\limits_{2^{m-1}\leq k<2^{m}}\left( k\left\vert K_{k}\left( x\right)
\right\vert \right) \\
&&+\frac{2}{2^{N}Q_{n}}\sum_{j=N+1}^{n_{1}}q_{2^{j-1}}\sum\limits_{m=1}^{N}%
\sup\limits_{2^{m-1}\leq k<2^{m}}\left( k\left\vert K_{k}\left( x\right)
\right\vert \right) \\
&&+\frac{2}{Q_{n}}\sum_{j=N+1}^{n_{1}}q_{2^{j-1}}\sum\limits_{m=N+1}^{j}\int%
\limits_{I_{N}}\sup\limits_{2^{m-1}\leq k<2^{m}}\left( k\left\vert
K_{k}\left( x\dotplus t\right) \right\vert \right) dt \\
&:&=J_{1}+J_{2}+J_{3}.
\end{eqnarray*}

Since $x\dotplus t\notin I_{N}$ by (\ref{fest}) and (\ref{SW}), we have $%
\left( m>N\right) $

\begin{eqnarray*}
&&\int\limits_{I_{N}}\sup\limits_{2^{m-1}\leq k<2^{m}}\left( k\left\vert
K_{k}\left( x\dotplus t\right) \right\vert \right) dt \\
&\leq &\sum\limits_{s=0}^{m}2^{s}\int\limits_{I_{N}}K_{2^{s}}\left(
x\dotplus t\right) dt \\
&=&\frac{1}{2^{N}}\sum\limits_{s=0}^{N}2^{s}K_{2^{s}}\left( x\right)
+\sum\limits_{s=N+1}^{m}2^{s}\int\limits_{I_{N}}K_{2^{s}}\left( x\dotplus
t\right) dt \\
&\leq &\frac{1}{2^{N}}\sum\limits_{s=0}^{N}2^{s}K_{2^{s}}\left( x\right)
\end{eqnarray*}%
\begin{eqnarray*}
&&+\sum\limits_{s=N+1}^{m}\sum\limits_{l=0}^{s}2^{l}\int%
\limits_{I_{N}}D_{2^{s}}\left( x\dotplus t\dotplus e_{l}\right) dt \\
&\leq &\frac{1}{2^{N}}\sum\limits_{s=0}^{N}2^{s}K_{2^{s}}\left( x\right) +%
\frac{2^{m}}{2^{N}}\sum\limits_{l=0}^{N-1}2^{l}\mathbf{1}_{I_{N}\left(
e_{l}\right) }\left( x\right) .
\end{eqnarray*}%
Consequently,%
\begin{eqnarray*}
J_{3} &\leq &\frac{2}{2^{N}Q_{n}}\sum_{j=N+1}^{n_{1}}q_{2^{j-1}}\left(
j-N\right) \sum\limits_{s=0}^{N}2^{s}K_{2^{s}}\left( x\right) \\
&&+\frac{2}{2^{N}Q_{n}}\sum_{j=N+1}^{n_{1}}q_{2^{j-1}}2^{j}\sum%
\limits_{l=0}^{N-1}2^{l}\mathbf{1}_{I_{N}\left( e_{l}\right) }\left( x\right)
\\
&=&\frac{2}{2^{2N}Q_{n}}\sum_{j=N+1}^{n_{1}}q_{2^{j-1}}2^{j}\frac{\left(
j-N\right) }{2^{j-N}}\sum\limits_{s=0}^{N}2^{s}K_{2^{s}}\left( x\right) \\
&&+\frac{2}{2^{N}Q_{n}}\sum_{j=N+1}^{n_{1}}q_{2^{j-1}}2^{j}\sum%
\limits_{l=0}^{N-1}2^{l}\mathbf{1}_{I_{N}\left( e_{l}\right) }\left(
x\right) .
\end{eqnarray*}

Since%
\begin{equation*}
Q_{n}\geq
\sum\limits_{j=1}^{2^{n_{1}}-1}q_{j}=\sum\limits_{r=1}^{n_{1}}\sum%
\limits_{j=2^{r-1}}^{2^{r}-1}q_{j}\geq
\sum\limits_{r=1}^{n_{1}}q_{2^{r}}2^{r-1}.
\end{equation*}%
We can write%
\begin{equation}
J_{3}\leq \frac{c}{2^{2N}}\sum\limits_{s=0}^{N}2^{s}K_{2^{s}}\left( x\right)
+\frac{c}{2^{N}}\sum\limits_{l=0}^{N-1}2^{l}\mathbf{1}_{I_{N}\left(
e_{l}\right) }\left( x\right) ,  \label{J3}
\end{equation}%
\begin{eqnarray}
J_{2} &\leq &\frac{c}{2^{2N}Q_{n}}\sum_{j=N+1}^{n_{1}}2^{j}q_{2^{j-1}}\sum%
\limits_{m=1}^{N}\sup\limits_{2^{m-1}\leq k<2^{m}}\left( k\left\vert
K_{k}\left( x\right) \right\vert \right)  \label{J2} \\
&\leq &\frac{c}{2^{2N}}\sum\limits_{m=1}^{N}\sup\limits_{2^{m-1}\leq
k<2^{m}}\left( k\left\vert K_{k}\left( x\right) \right\vert \right) .  \notag
\end{eqnarray}%
By Lemma \ref{m(p>1/2)} and from the condition of theorem we can write%
\begin{equation}
\int\limits_{\overline{I}_{N}}\sup\limits_{n>N}\left(
2^{N/p}\int\limits_{I_{N}}\left\vert F_{n,2}^{\left( 1\right) }\left(
x\dotplus t\right) \right\vert dt\right) ^{p}dx  \label{Fn1}
\end{equation}%
\begin{equation*}
\leq \frac{c_{p}2^{N\left( 1-p\right) }}{Q_{n}^{p}}%
\sum_{j=1}^{N}q_{2^{j-1}}^{p}\sum\limits_{m=1}^{j}\int\limits_{\overline{I}%
_{N}}\sup\limits_{2^{m-1}\leq k<2^{m}}\left( k\left\vert K_{k}\left(
x\right) \right\vert \right) ^{p}dx
\end{equation*}%
\begin{equation*}
+\frac{c_{p}2^{N}}{2^{2pN}}\sum\limits_{m=1}^{N}\int\limits_{\overline{I}%
_{N}}\sup\limits_{2^{m-1}\leq k<2^{m}}\left( k\left\vert K_{k}\left(
x\right) \right\vert \right) ^{p}dx
\end{equation*}%
\begin{equation*}
+c_{p}2^{N\left( 1-2p\right) }\sum\limits_{s=0}^{N-1}\int\limits_{\overline{I%
}_{N}}\left( 2^{s}K_{2^{s}}\left( x\right) \right) ^{p}dx
\end{equation*}%
\begin{equation*}
+c_{p}2^{N\left( 1-p\right) }\sum\limits_{l=0}^{N-1}2^{lp}\int\limits_{%
\overline{I}_{N}}\mathbf{1}_{I_{N}\left( e_{l}\right) }\left( x\right) dx
\end{equation*}%
\begin{equation*}
\leq \frac{c_{p}2^{N\left( 1-p\right) }}{Q_{n}^{p}}%
\sum_{j=1}^{N}q_{2^{j-1}}^{p}2^{j\left( 2p-1\right) }
\end{equation*}%
\begin{equation*}
+c_{p}2^{N\left( 1-2p\right) }\sum\limits_{s=0}^{N-1}2^{s\left( 2p-1\right) }
\end{equation*}%
\begin{equation*}
+c_{p}\frac{2^{N\left( 1-p\right) }}{2^{N}}\sum\limits_{l=0}^{N-1}2^{lp}
\end{equation*}%
\begin{equation*}
\leq \frac{c_{p}2^{N\left( 1-p\right) }}{Q_{n}^{p}}\sum_{j=1}^{N}\left(
q_{2^{j-1}}2^{j}\right) ^{p}2^{j\left( p-1\right) }+c_{p}
\end{equation*}%
\begin{eqnarray*}
&\leq &c_{p}\sup\limits_{N\in \mathbb{N}}\frac{2^{N\left( 1-p\right) }}{%
Q_{2^{N}}^{p}}\sum_{j=1}^{N}Q_{2^{j}}^{p}2^{j\left( p-1\right) }+c_{p} \\
&\leq &c_{p}<\infty .
\end{eqnarray*}

Analogously, we can prove that 
\begin{equation}
\int\limits_{\overline{I}_{N}}\sup\limits_{n>N}\left(
2^{N/p}\int\limits_{I_{N}}\left\vert F_{n,2}^{\left( 2\right) }\left(
x\dotplus t\right) \right\vert dt\right) ^{p}dx\leq c_{p}<\infty .
\label{Fn2}
\end{equation}

Combine (\ref{!+2}), (\ref{1-2}), (\ref{F1}), (\ref{Fn1}) and (\ref{Fn2}) we
complete the proof of Theorem \ref{main}.
\end{proof}

For $p=1,$ Theorem \ref{main} implies that the following two conditions are
equivalent:\newline

\begin{itemize}
\item The maximal operator $t^{\ast }$ is bounded from the dyadic Hardy
space $H_{1}\left( \mathbb{I}\right) $ to the space $L_{1}\left( \mathbb{I}%
\right) ;$\newline

\item $\sup\limits_{N\in \mathbb{N}}\frac{1}{Q_{2^{N}}}\sum%
\limits_{j=1}^{N}Q_{2^{j}}<\infty .$
\end{itemize}

On the other hand, in \cite{goginava2022some} it is proved that the
following two conditions are equivalent:\newline

\begin{itemize}
\item The maximal operator $t^{\ast }$ is bounded from the space $L_{\infty
}\left( \mathbb{I}\right) $ to the space $L_{\infty }\left( \mathbb{I}%
\right) ;$

\item $\sup\limits_{N\in \mathbb{N}}\frac{1}{Q_{2^{N}}}\sum%
\limits_{j=1}^{N}Q_{2^{j}}<\infty .$
\end{itemize}

Hence, we can conclude that the following.

\begin{theorem}
\label{three}The following three conditions are equivalent:\newline
\end{theorem}

\begin{itemize}
\item The maximal operator $t^{\ast }\left( f\right) $ is bounded from $%
L_{\infty }(\mathbb{I})$ to $L_{\infty }(\mathbb{I});$

\item The maximal operator $t^{\ast }\left( f\right) $ is bounded from $%
H_{1}(\mathbb{I})$ to $L_{1}(\mathbb{I});$

\item $\sup\limits_{n\in \mathbb{N}}\frac{1}{Q_{n}}%
\sum_{k=1}^{|n|}Q_{2^{k}}<\infty .$
\end{itemize}

\section{Applications to various summability methods}

Since, the Nörlund mean is a generalization of many other well-know means
with wide range of literature, in the last section we give applications of
our results.

\begin{example}
Fej\'er means: Let $q_{j}=1.$ then it is easy to see that $Q_{j}\sim j$ and
we have%
\begin{eqnarray*}
\frac{2^{N\left( 1-p\right) }}{Q_{2^{N}}^{p}}\sum%
\limits_{j=1}^{N}Q_{2^{j}}^{p}2^{j\left( p-1\right) } &=&\frac{2^{N\left(
1-p\right) }}{2^{Np}}\sum\limits_{j=1}^{N}2^{jp}2^{j\left( p-1\right) } \\
&=&2^{N\left( 1-2p\right) }\sum\limits_{j=1}^{N}2^{j\left( 2p-1\right) }.
\end{eqnarray*}%
Hence,%
\begin{equation*}
\left( \sup\limits_{N\in \mathbb{N}}\frac{2^{N\left( 1-p\right) }}{%
Q_{2^{N}}^{p}}\sum\limits_{j=1}^{N}Q_{2^{j}}^{p}2^{j\left( p-1\right)
}<\infty \right) \iff \left( p>1/2\right)
\end{equation*}%
and we have that the following two conditions are equivalent:\newline
\end{example}

\begin{itemize}
\item The maximal operator $\sup\limits_{n\in \mathbb{N}}\left\vert \sigma
_{n}\left( f\right) \right\vert $ is bounded from the dyadic Hardy space $%
H_{p}\left( \mathbb{I}\right) $ to the space $L_{p}\left( \mathbb{I}\right)
; $

\item $p>1/2.$\newline
\newline
Under the condition $p>1/2,$ the boundedness of maximal operator $%
\sup\limits_{n\in \mathbb{N}}\left\vert \sigma _{n}\left( f\right)
\right\vert $ was proved by Weisz \cite{weisz1996cesaro}, and the essence of
condition $p>1/2$ was proved by the author \cite{GogiEJA}.
\end{itemize}

\begin{example}
$\left( C,\alpha \right) $-means: Let $q_{j}:=A_{j}^{\alpha -1},\alpha \in
\left( 0,1\right) $. It is easy to see that $Q_{j}\sim j^{\alpha }$. Since%
\begin{eqnarray*}
\frac{2^{N\left( 1-p\right) }}{Q_{2^{N}}^{p}}\sum%
\limits_{j=1}^{N}Q_{2^{j}}^{p}2^{j\left( p-1\right) } &=&\frac{2^{N\left(
1-p\right) }}{Q_{2^{N}}^{p}}\sum\limits_{j=1}^{N}2^{jp\alpha }2^{j\left(
p-1\right) } \\
&=&\frac{2^{N\left( 1-p\right) }}{Q_{2^{N}}^{p}}\sum\limits_{j=1}^{N}2^{j%
\left( p\left( \alpha +1\right) -1\right) }
\end{eqnarray*}%
we conclude that 
\begin{equation*}
\sup\limits_{N\in \mathbb{N}}\frac{2^{N\left( 1-p\right) }}{Q_{2^{N}}^{p}}%
\sum\limits_{j=1}^{N}Q_{2^{j}}^{p}2^{j\left( p-1\right) }<\infty \iff \left(
p>\frac{1}{1+\alpha }\right) .
\end{equation*}%
Consequently, we have the following two conditions are equivalent:\newline
\end{example}

\begin{itemize}
\item The maximal operator $\sup\limits_{n\in \mathbb{N}}\left\vert \sigma
_{n}^{\alpha }\left( f\right) \right\vert $ is bounded from the dyadic Hardy
space $H_{p}\left( \mathbb{I}\right) $ to the space $L_{p}\left( \mathbb{I}%
\right) ;$

\item $p>\frac{1}{1+\alpha }.$\newline
\newline
Under the condition $p>\frac{1}{1+\alpha },$ the boundedness of maximal
operator $\sup\limits_{n\in \mathbb{N}}\left\vert \sigma _{n}^{\alpha
}\left( f\right) \right\vert $ was proved by Weisz \cite{weisz1996cesaro},
and the importance of condition $p>\frac{1}{1+\alpha }$ was proved by the
author \cite{Buda}.
\end{itemize}

\begin{example}
Let $q_{j}:=j^{\alpha -1},\alpha \in \lbrack 0,1)$. First, we consider the
case when $\alpha =0$. Then the Nörlund means coincide to the Nörlund
logarithmic means 
\begin{equation*}
t_{n}(f;x):=\frac{1}{Q_{n}}\sum_{k=1}^{n-1}\frac{S_{k}(f;x)}{n-k}.
\end{equation*}

Nörlund's logarithmic means with respect to the trigonometric system was
studied by Tkebuchava \cite{TkebuchavaAMAPN,TkebuchavaBull}. The convergence
and divergence of this means with respect to the Walsh systems was discussed
in \cite{GatGogiAMSDiv1,GatGogiAMSDiv2,GatGogiTkebJMAA,GoginavaMisc}. Since%
\begin{equation*}
\sup\limits_{n\in \mathbb{N}}\frac{1}{Q_{n}}\sum_{k=1}^{|n|}Q_{2^{k}}\sim
\sup\limits_{n\in \mathbb{N}}\frac{\left\vert n\right\vert ^{2}}{\log \left(
n+1\right) }\sim \sup\limits_{n\in \mathbb{N}}\log \left( n+1\right) =\infty
\end{equation*}%
from Theorem we conclude that the maximal operator $\sup\limits_{n\in 
\mathbb{N}}\left\vert t_{n}(f)\right\vert $ is not bounded from $H_{1}(%
\mathbb{I})$ to $L_{1}(\mathbb{I})$ and consequently, by interpolation
theorem can not be bounded from $H_{p}(\mathbb{I})$ to $L_{p}(\mathbb{I})$,
when $p<1$.

Now, we suppose that $\alpha \in (0,1)$. It is easy to see that%
\begin{eqnarray*}
&&\frac{2^{N\left( 1-p\right) }}{Q_{2^{N}}^{p}}\sum%
\limits_{j=1}^{N}Q_{2^{j}}^{p}2^{j\left( p-1\right) } \\
&=&\frac{2^{N\left( 1-p\right) }}{Q_{2^{N}}^{p}}\sum\limits_{j=1}^{N}2^{j%
\left( p\left( \alpha +1\right) -1\right) }
\end{eqnarray*}%
and%
\begin{equation}
\lim\limits_{n\rightarrow \infty }\frac{1}{n^{\alpha }}\sum_{k=1}^{n}\left(
n-k\right) ^{\alpha -1}S_{k}(f;x)=f\left( x\right) \text{ \ for a. e. }x\in 
\mathbb{I}.  \label{c,a}
\end{equation}%
we have the following two conditions are equivalent:\newline
\end{example}

\begin{itemize}
\item The maximal operator $\sup\limits_{n\in \mathbb{N}}\frac{1}{n^{\alpha }%
}\left\vert \sum_{k=1}^{n}\left( n-k\right) ^{\alpha
-1}S_{k}(f;x)\right\vert $ is bounded from the dyadic Hardy space $%
H_{p}\left( \mathbb{I}\right) $ to the space $L_{p}\left( \mathbb{I}\right)
; $

\item $p>\frac{1}{1+\alpha }.$\newline
\end{itemize}

\section{Declaration}

The author declare that he has not conflict of interest.

\section{Data Availability}

The author confirm that he does not use any data.

\end{document}